\documentclass[11pt,reqno]{amsart}
\usepackage{graphicx}
\usepackage{amssymb}
\usepackage{amsmath}
\usepackage{url}

\newtheorem{thm}{Theorem}

\newtheorem{lem}[thm]{Lemma}
\newtheorem{prop}[thm]{Proposition}

\newtheorem{ques}{Question}

\theoremstyle{definition}
\newtheorem{defn}{Definition}

\theoremstyle{remark}
\newtheorem{rem}{Remark}[section]
\numberwithin{equation}{section}
\def\R{\mathbb R}

\def\f{\frac}

\def\la{\langle}
\def\ra{\rangle}
\def\pt{\partial}

\begin{document}
\title[Gauss-Bonnet-Chern mass for  graphic manifolds]{The Gauss-Bonnet-Chern mass for  graphic manifolds}

\author{Haizhong Li}
\address{Department of mathematical sciences, and Mathematical Sciences
Center, Tsinghua University, 100084, Beijing, P. R. China}
\email{hli@math.tsinghua.edu.cn}
\author{Yong Wei}
\address{Department of mathematical sciences, Tsinghua University, 100084, Beijing, P. R. China}
\email{wei-y09@mails.tsinghua.edu.cn}
\author{Changwei Xiong}
\address{Department of mathematical sciences, Tsinghua University, 100084, Beijing, P. R. China}
\email{xiongcw10@mails.tsinghua.edu.cn}
\thanks{The research of the authors was supported by NSFC No. 11271214.}

\begin{abstract}
In this paper, we prove a positive mass theorem and Penrose-type inequality of the Gauss-Bonnet-Chern mass $m_2$ for the graphic manifold with flat normal bundle.
\end{abstract}
\maketitle

\section{Introduction}

A complete manifold $(M^n,g)$ is said to be asymptotically flat of order $\tau$ if there is a compact set $K$ such that $M\setminus K$ is diffeomorphic to $\R^n\setminus B_R(0)$ for some $R>0$ such that in this coordinate chart, the metric $g_{ij}(x)$ satisifes
\begin{align*}
  &g_{ij}(x) = \delta_{ij}+\sigma_{ij}, \\
   & |\sigma_{ij}|+|x||\pt_k\sigma_{ij}|+|x|^2|\pt_k\pt_l\sigma_{ij}| =O(|x|^{-\tau}).
\end{align*}
The ADM mass was introduced by Arnowitt, Deser and Misner \cite{ADM} for asymptotically flat manifold.
 \begin{align}\label{def-adm}
 m_1=m_{ADM}:=\frac 1{2(n-1)\omega_{n-1}}\lim_{r\rightarrow\infty}\int_{S_r}(g_{ij,i}-g_{ii,j})\nu_jdS,
 \end{align}
where $\omega_{n-1}$ is the area of the standard $(n-1)$ dimensional unit sphere, $S_r$ is the Euclidean coordinate sphere, $dS$ is the volume element on $S_r$ induced by the Euclidean metric and $\nu=r^{-1}x$ is the outward normal of $S_r$ in $\R^n$.  Bartnik \cite{Ba} proved that the ADM mass is well-defined and is a geometric invariant for asymptotically flat manifold provided the order $\tau>\frac {n-2}2$.

The positive mass theorem and Penrose inequality for ADM mass are two important results in differential geometry and general relativity. The positive mass theorem states that the ADM mass is nonnegative on any asymptotically flat manifold with order $\tau>\frac{n-2}2$ and with nonnegative scalar curvature, which was proved by Schoen-Yau \cite{SY79,SY80} for $3\leq n\leq 7$, Schoen-Yau \cite{SY88} for conformally flat manifold, Witten \cite{Wit} for spin manifold, Lam \cite{Lam} for codimension one graphic manifold and Mirandola-Vit\'{o}rio \cite{MV} for arbitrary codimension graphic manifold with flat normal bundle.  The Penrose inequality can be viewed as a generalization of the positive mass theorem in the presence of an area minimizing horizon. When the asymptotically flat manifold has nonnegative scalar curvature and has order $\tau>\frac {n-2}2$, the Penrose inequality gives the lower bound for the ADM mass of $M$ in terms of the area of the horizon $\Sigma$, precisely
\begin{align*}
  m_{ADM} \geq& \frac 12(\frac{|\Sigma|}{\omega_{n-1}})^{\frac{n-2}{n-1}}.
\end{align*}
The equality holds if and only if $M$ is isometric to the Schwarzschild metric. The Penrose inequality was proved by Huisken-Ilmanen \cite{HI} for $n=3$ and $\Sigma$ is connected by using inverse mean curvature flow, Bray \cite{Br} for $n=3$ and $\Sigma$ has multiple components by using conformal flow, Bray-Lee \cite{BL} for $3\leq n\leq 7$, with the extra requirement that $M$ is spin for the rigidity statement. Lam \cite{Lam} also proved the Penrose inequality for codimension one graphic manifold. See also \cite{HW1,HW2,deLG1,deLG2} for related work. Recently, Mirandola-Vit\'{o}rio generalized Lam's result to arbitrary codimension graph with flat normal bundle. The readers can refer to \cite{Br1,Mars} for surveys on the the positive mass theorem and Penrose inequality.

Recall that the tensor $P_1$ given by
\begin{align*}
P_1^{ijkl}=\frac 12(g^{ik}g^{jl}-g^{il}g^{jk})
\end{align*}
is closely related to the scalar curvature $R$. In fact, we have $R=P_1^{ijkl}R_{ijkl}$. By using the tensor $P_1$, one can also define a mass (see \cite{GWW}), which is just the ADM mass with a slightly different but equivalent form
\begin{align}
m_1=&\frac 1{(n-1)\omega_{n-1}}\lim_{r\rightarrow \infty}\int_{S_r}P_1^{ijkl}\pt_lg_{jk}\nu_idS\nonumber\\
=&\frac 1{2(n-1)\omega_{n-1}}\lim_{r\rightarrow \infty}\int_{S_r}(g^{ik}\pt^jg_{jk}-g^{jk}\pt^ig_{jk})\nu_idS.\label{def-adm1}
\end{align}

Recently, Ge-Wang-Wu \cite{GWW} introduced a new mass, which they called Gauss-Bonnect-Chern mass $m_2$, by using the tensor $P_2$ (in the following, we will write $P_2$ as $P$ for simplicity).
\begin{equation}\label{def-P}
 P^{ijkl}=R^{ijkl}+R^{jk}g^{il}-R^{jl}g^{ik}-R^{ik}g^{jl}+R^{il}g^{jk}+\frac 12R(g^{ik}g^{jl}-g^{il}g^{jk}).
\end{equation}
Recall that the second Gauss-Bonnet curvature $L_2$ satisfies $L_2=P_2^{ijkl}R_{ijkl}$.

\begin{defn}[\cite{GWW}] Let $n\geq 5$. Suppose $(M^n,g)$ is an asymptotically flat manifold of decay order $\tau>\f{n-4}{3}$, and the second Gauss-Bonnet curvature given by $L_2=P^{ijkl}R_{ijkl}$ is integrable. Then the second Gauss-Bonnet-Chern mass $m_2$ is defined as
\begin{equation}
m_2(g)=c_2(n)\lim_{r\rightarrow \infty} \int_{S_r}P^{ijkl}\pt_l g_{jk}\nu_i dS,
\end{equation}
where $c_2(n)=\f{1}{2(n-1)(n-2)(n-3)\omega_{n-1}}$.
\end{defn}

In the same paper \cite{GWW}, Ge-Wang-Wu proved that $m_2$ is well-defined and is a geometric invariant provided $(M,g)$ is asymptotically flat of order $\tau>\frac {n-4}3$. Then they asked the following natural question:
 \begin{ques}[\cite{GWW}]Is the mass $m_2$ nonnegative when the Gauss-Bonnet curvature $L_2$ is nonnegative? Does the Penrose inequality for $m_2$ also hold?
 \end{ques}
 In \cite{GWW}, Ge-Wang-Wu give an affirmative answer to these questions for codimension one graphic manifolds. In the second paper, Ge-Wang-Wu \cite{GWW2} answer the questions for conformally flat manifolds.

Inspired by Mirandola-Vit\'{o}rio's new paper \cite{MV}, we now consider the positive mass theorem and Penrose inequality of $m_2$ for arbitrary codimension graphic manifold. To state our theorems, we first give the following definition.

\begin{defn}
Let $f:\R^n\rightarrow \R^m$ be a smooth map and let $f_i^{\alpha},f_{ij}^{\alpha},f_{ijk}^{\alpha}$ denote the first,second and third derivatives of $f$, where $1\leq i,j,k\leq n$ and $1\leq \alpha\leq m$. $f$ is called asymptotically flat of order $\tau$ if
\begin{align*}
  |f_i^{\alpha}(x)|+|f_{ij}^{\alpha}(x)||x|+|f_{ijk}^{\alpha}(x)||x|^2=O(|x|^{-\tau/2})
\end{align*}
at infinity for some $\tau> (n-4)/3$.
\end{defn}
Since the induced metric on the graph $M=\{(x,f(x))|x\in\R^n\}$ (see section 2 for detail) is $g_{ij}=\delta_{ij}+f^{\alpha}_if^{\alpha}_j$. We can easily check that $M$ is asymptotically flat of order $\tau$.

\begin{thm}[Positive mass theorem]\label{thm2}
Let $M^n\subset \R^{n+m}$ be the graph of an asymptotically flat map $f:\R^n\rightarrow \R^m$ endowed with the natural metric. Then the second Gauss-Bonnet-Chern mass $m_2$ of $M$ is equal to
\begin{equation}
m_2(g)=\f{1}{2}c_2(n)\int_M(L_2+L_2^{\perp})\f{1}{\sqrt{G}}d\mu_M,
\end{equation}
where $L_2^{\perp}$ is defined  by \eqref{eq-L2} and it vanishes provided that $M$ has flat normal bundle. In particular, if $M$ has flat normal bundle and nonnegative $L_2$ curvature, then $m_2$ is nonnegative.
\end{thm}

The second part of this paper considers the Penrose inequality for $m_2$ in the graphic case. We consider asymptotically flat manifold containing an area outer minimizing horizon. A horizon is simply a minimal surface, which is area outer minimizing if every other surface enclosing it has greater area. We have
\begin{thm}\label{thm3}
Let $\Omega\subset \R^n$ be a bounded open subset with Lipschitz boundary $\pt \Omega$. Let $f:\R^n\setminus \Omega \rightarrow \R^m$ be an asymptotically flat map. Assume that $f$ is constant along each connected component of $\Sigma=\pt \Omega$. Let $M$ be the graph of $f$ with its natural metric. Then
\begin{equation*}
m_2=\f{1}{2}c_2(n)\int_M (L_2+L_2^{\perp})\f{1}{\sqrt{G}}d\mu_M+3c_2(n)\int_\Sigma (\f{|Df|^2}{1+|Df|^2})^2 H_3 d\mu_{\Sigma},
\end{equation*}
where $|Df|^2=|Df^1|^2+\cdots+|Df^m|^2$ and $H_3$ is the $3$-th mean curvature of the hypersurface $\Sigma$ in $\R^n$.
\end{thm}

If $\Sigma$ is in the level set of $f$, then $\Sigma$ can be identified with its graph $\{(x,f(x))|x\in\Sigma\}$. Then the mean curvature $\hat{H}$ of $f(\Sigma)$ in $(M,g)$ and the mean curvature $H$ of $\Sigma$ in $\R^n$ are related by (see section 4)
\begin{align*}
\hat{H}=\frac 1{\sqrt{1+|Df|^2}}H.
\end{align*}
Therefore if $|Df(x)|\rightarrow \infty$ as $x\rightarrow \Sigma=\pt\Omega$, the graph of $\Sigma$ is a horizon in $(M,g)$.

\begin{thm}\label{thm4}
Let $\Omega\subset\R^n$ be an open subset. Let $f:\R^n \setminus \Omega \rightarrow \R^m$ be a continuous map that is constant along each connected component of the boundary $\Sigma=\pt \Omega$ and asymptotically flat in $\R^n\setminus \bar{\Omega}$. Assume that the graph $M$ of $f$ extends $C^2$ to its boundary $\pt M$. Assume further that along each connected component $\Sigma_i$ of $\pt M$, the manifold $\bar{M}$ is tangent to the cylinder $\Sigma \times l_i$, where $l_i$ is a straight line of $\R^m$. Then
\begin{equation*}
m_2=\f{1}{2}c_2(n)\int_M (L_2+L_2^{\perp})\f{1}{\sqrt{G}}d\mu_M+3c_2(n)\int_\Sigma  H_3 d\mu_{\Sigma},
\end{equation*}
where $H_3$ is the $3$-th mean curvature of the embedding $\Sigma$ in $\R^n$.
\end{thm}

The proof of Theorem \ref{thm4} is just by checking $|Df(x)|\rightarrow\infty$ as $x\rightarrow\Sigma$.

A hypersurface $\Sigma\subset\R^n$ is called $3$-convex if its mean curvature $H$, $2$-th mean curvature $H_2$ and $3$-th mean curvature $H_3$ are positive on $\Sigma$. As in \cite{Lam,GWW,HW2,MV}, by applying the Alexandrov-Fenchel inequality due to Guan-Li \cite{GL} to Theorem \ref{thm4}, we conclude that
\begin{thm}[Penrose inequality]\label{thm-penrose}
Under the hypothesis of Theorem \ref{thm4}, and we assume that $M$ has nonnegative $L_2$ and flat normal bundle. If each connected component of the boundary $\Sigma=\pt \Omega$ is $3$-convex and star-shaped, then
\begin{align}\label{eq-penrose}
  m_2\geq \frac 14(\frac {|\Sigma|}{\omega_{n-1}})^{\frac {n-4}{n-1}}.
\end{align}
Moreover, the equality implies $L_2$ vanishes identically on $M$ and $\Sigma$ is a sphere.
\end{thm}

The rest of this paper is organized as follows. In section 2, we give some notations and basic formulas. In section 3, we first prove proposition \ref{prop-mass} which is the key ingredient to prove our main theorems. The crucial step of proving proposition \ref{prop-mass} is checking that $\frac 12L_2^{\bot}$ can be represented as normal curvature terms. In the sections 4 and 5, we will use the similar idea of Ge-Wang-Wu \cite{GWW} to complete the proof of  positive mass theorem and Penrose inequality for $m_2$. In the last section, we give a generalization of Ge-Wang-Wu's  \cite{GWW1} result on Positive mass theorem and Penrose inequality in the Einstein-Gauss-Bonnet gravity to arbitrary codimension graph.

\section{Preliminaries}

Recall that on a complete Riemannian manifold $(M,g)$, the Gauss-Bonnet curvature $L_2$ is defined as
\begin{align}
    L_2=P^{ijkl}R_{ijkl},
\end{align}
where the (0,4) tensor $P$ is given by \eqref{def-P}. Note that $P$ has the same symmetric property as the Riemannian curvature tensor, i.e.,
\begin{align}\label{P-symmetri}
    P^{ijkl}=-P^{jikl}=-P^{ijlk}=P^{klij}.
\end{align}
$P$ also satisfies the divergence-free property (cf. \cite{GWW,Da}), i.e.,
\begin{align}\label{P-divfree}
    \nabla_iP^{ijkl}=\nabla_jP^{ijkl}=\nabla_kP^{ijkl}=\nabla_lP^{ijkl}=0.
\end{align}

We now suppose $M=\{(x,f(x))|x\in\R^n\}$ is the graph of a smooth asymptotically flat map $f:\R^n\rightarrow\R^m$. The  vectors  $\pt_i=(e_i,f_i^{\alpha}e_{\alpha})$ are tangential and the vectors $\eta^{\alpha}=(-Df^{\alpha},e_{\alpha})$ are normal to $M$. In this paper, we will use the convention on the indices:
\begin{align*}
    1\leq i,j,k,l,\cdots\leq n\qquad \textrm{and} \quad 1\leq \alpha,\beta,\gamma,\mu,\nu,\cdots\leq m.
\end{align*}
The induced metric on $M$ is
\begin{align*}
    g_{ij}=\delta_{ij}+f_i^{\alpha}f_j^{\alpha}.
\end{align*}
We will denote the matrix $U$ by (cf.\cite{MV})
\begin{align*}
    U_{\alpha\beta}=\langle \eta^{\alpha},\eta^{\beta}\rangle=\delta_{\alpha\beta}+\langle Df^{\alpha},Df^{\beta}\rangle,
\end{align*}
with the inverse matrix $U^{\alpha\beta}$. Then
\begin{align*}
    &\delta_{\gamma\beta}=U^{\gamma\beta}+U^{\alpha\gamma}\langle Df^{\alpha},Df^{\beta}\rangle,\\
    &g^{ij}=\delta_{ij}-f_i^{\alpha}f_j^{\beta}U^{\alpha\beta}.
\end{align*}
Denote $f^{\alpha}, \alpha=1,\cdots,m$ the components of $f$. The gradient of $f^{\alpha}:M\rightarrow\R$ is
\begin{align*}
    &\nabla f^{\alpha}=g^{jk}f_k^{\alpha}\pt_j=U^{\alpha\gamma}f_i^{\gamma}\pt_i.
\end{align*}
The shape operator $A^{\alpha}$ with respect to $\eta^{\alpha}$
\begin{align}
    A^{\alpha}\pt_i=-(\bar{\nabla}_{\pt_i}\eta^{\alpha})^{\top}=f^{\alpha}_{ik}g^{kj}\pt_j.
\end{align}
Therefore
\begin{align}
    \langle A^{\mu}\pt_i,A^{\gamma}\pt_k\rangle=g^{lr}f_{il}^{\mu}f_{kr}^{\gamma}\quad \textrm{and} \quad f^{\alpha}_{ik}=\langle A^{\alpha}\pt_i,\pt_k\rangle.
\end{align}
The shape operator is symmetric:
\begin{align}\label{eq-shape}
    \langle A^{\alpha}X,Y\rangle=\langle X,A^{\alpha}Y\rangle, \quad 1\leq \alpha\leq m
\end{align}
for any tangent $X,Y$. The second fundamental form
\begin{equation}
B(\pt_i,\pt_j)=\langle A^{\alpha}\pt_i,\pt_j\rangle U^{\alpha\beta}\eta^{\beta}=f^{\alpha}_{ij}U^{\alpha\beta}\eta^{\beta}.
\end{equation}

\begin{lem}\label{lem-formulas}
The curvature tensor $R_{ijkl}$, the Gauss-Bonnet curvature $L_2$ and the Christoffel symbols $\Gamma_{ij}^k$ have the following expressions:
\begin{align*}
    R_{ijkl}=&U^{\alpha\beta}(f^{\alpha}_{ik}f^{\beta}_{jl}-f^{\alpha}_{il}f^{\beta}_{jk}),\\
    L_2=&P^{ijkl}R_{ijkl}=2P^{ijkl}U^{\alpha\beta}f^{\alpha}_{ik}f^{\beta}_{jl},\\
    \Gamma_{ij}^k=&U^{\alpha\beta}f_k^{\alpha}f_{ij}^{\beta}.
\end{align*}
\end{lem}
\proof
By Gauss equation, we have
\begin{align*}
    R_{ijkl}=&\langle B(\pt_i,\pt_k),B(\pt_j,\pt_l)\rangle -\langle B(\pt_i,\pt_l),B(\pt_j,\pt_k)\rangle \\
    =&U^{\alpha\beta}(f^{\alpha}_{ik}f^{\beta}_{jl}-f^{\alpha}_{il}f^{\beta}_{jk}).
\end{align*}
The Gauss-Bonnet curvature $L_2$
\begin{align*}
    L_2=P^{ijkl}R_{ijkl}=&P^{ijkl}U^{\alpha\beta}(f^{\alpha}_{ik}f^{\beta}_{jl}-f^{\alpha}_{il}f^{\beta}_{jk})\nonumber\\
    =&2P^{ijkl}U^{\alpha\beta}f^{\alpha}_{ik}f^{\beta}_{jl}.
\end{align*}
By the definition of $\Gamma_{ij}^k$,
\begin{align*}
    \Gamma_{ij}^k=&\frac 12g^{kl}(\pt_ig_{jl}+\pt_jg_{il}-\pt_lg_{ij})\\
    =&g^{kl}f^{\gamma}_{ij}f^{\gamma}_l\\
    =&(\delta_{kl}-f_k^{\alpha}f_l^{\beta}U^{\alpha\beta})f^{\gamma}_{ij}f^{\gamma}_l\\
    =&f^{\gamma}_{ij}f^{\gamma}_k-f_k^{\alpha}f^{\gamma}_{ij}f^{\gamma}_lf_l^{\beta}U^{\alpha\beta}\\
    =&f^{\gamma}_{ij}f^{\gamma}_k-f_k^{\alpha}f^{\gamma}_{ij}(\delta_{\alpha\gamma}-U^{\alpha\gamma})\\
    =&f_k^{\alpha}f^{\gamma}_{ij}U^{\alpha\gamma}.
\end{align*}
\endproof

The commutation of shape operators gives the normal curvature operator $R^{\bot}_{\alpha\beta}:TM\rightarrow TM$,
\begin{align}\label{eq-normal}
  R^{\bot}_{\alpha\beta}(X) =& (A^{\beta}A^{\alpha}-A^{\alpha}A^{\beta})X.
\end{align}
Thus,
\begin{align*}
    \langle R^{\bot}_{\alpha\beta}(X),Y\rangle =&\langle (A^{\beta}A^{\alpha}-A^{\alpha}A^{\beta})X,Y\rangle\\
    =&\langle A^{\alpha}X,A^{\beta}Y\rangle-\langle A^{\beta}X,A^{\alpha}Y\rangle
\end{align*}
for any tangent $X,Y$ and any $1\leq \alpha,\beta\leq m$. The Ricci operator $Rc:TM\rightarrow TM$ is given by
\begin{align}\label{eq-Ricci}
 \langle Rc(X),Y\rangle=&Ric (X,Y),
\end{align}
which is symmetric in $X$ and $Y$. We have the following lemma about the commutation of the Ricci operator $Rc$ and shape operator $A^{\alpha}$.
\begin{lem}
The Ricci operator $Rc$ and shape operator $A^{\alpha}$ commutes as following:
\begin{equation}\label{eq-commute}
  A^{\alpha} Rc-Rc A^{\alpha}=U^{\mu\nu}((Tr A^{\mu})R^{\bot}_{\nu\alpha}+R^{\bot}_{\alpha\mu}A^{\nu}+A^{\mu} R^{\bot}_{\alpha\nu}).
\end{equation}
\end{lem}
\proof
By \eqref{eq-normal} and the definition of Ricci operator, we have
\begin{align*}
  A^{\alpha}Rc-Rc A^{\alpha} =& U^{\mu\nu}\left(A^{\alpha}((Tr A^{\mu})A^{\nu}-A^{\mu}A^{\nu})- ((Tr A^{\mu})A^{\nu}-A^{\mu}A^{\nu})A^{\alpha}\right)\\
  =& U^{\mu\nu}Tr A^{\mu}(A^{\alpha}A^{\nu}-A^{\nu}A^{\alpha})+U^{\mu\nu}(A^{\mu}A^{\nu}A^{\alpha}-A^{\alpha}A^{\mu}A^{\nu})\\
  =&U^{\mu\nu}(Tr A^{\mu})R^{\bot}_{\nu\alpha}+U^{\mu\nu}\left(A^{\mu}(A^{\alpha}A^{\nu}+R^{\bot}_{\alpha\nu})-A^{\alpha}A^{\mu}A^{\nu}\right)\\
  =&U^{\mu\nu}(Tr A^{\mu})R^{\bot}_{\nu\alpha}+U^{\mu\nu}(R^{\bot}_{\alpha\mu}A^{\nu}+A^{\mu}R^{\bot}_{\alpha\nu}),
\end{align*}
which gives the commutation formula \eqref{eq-commute}.
\endproof

In the rest of the paper, we will denote the right-hand side of \eqref{eq-commute} by $T^{\bot}_{\alpha}$.

\section{Positive mass theorem for $m_2$}
We first prove the following proposition, which is a key ingredient for proving our theorems.
\begin{prop}\label{prop-mass}
\begin{align}\label{key-equ1}
    \pt_i(P^{ijkl}\pt_lg_{jk})=&\frac 12L_2+ \frac 12L_2^{\bot},
\end{align}
where $ L_2^{\bot}$ can be expressed as some normal curvature terms (see \eqref{eq-L2} for the explicit expression). In particular, when the graph $M$ of $f$ has flat normal bundle, we have
\begin{align}\label{key-equ}
    \pt_i(P^{ijkl}\pt_lg_{jk})=&\frac 12L_2.
\end{align}
\end{prop}
\begin{rem}
The equation \eqref{key-equ} was shown by Ge-Wang-Wu in \cite[lemma 4.3]{GWW} for the hypersurface graph case, i.e., $f:\R^n\rightarrow\R$ and $M=\{(x,f(x))|x\in\R^n\}$ is a hypersurface in $\R^{n+1}$. Equation \eqref{key-equ} was the key ingredient to show the positive mass theorem for $m_2$ in the hypersurface graph case.
\end{rem}
\proof
\begin{align*}
    \pt_i(P^{ijkl}\pt_lg_{jk})=&\pt_iP^{ijkl}\pt_lg_{jk}+P^{ijkl}\pt_i\pt_lg_{jk}.
\end{align*}
By the symmetric property \eqref{P-symmetri} of $P^{ijkl}$, the second term is equal to:
\begin{align}
    P^{ijkl}\pt_i\pt_lg_{jk}=&P^{ijkl}(f^{\alpha}_{ijl}f^{\alpha}_k+f^{\alpha}_{jl}f^{\alpha}_{ik}+f^{\alpha}_{ji}f^{\alpha}_{kl}+f^{\alpha}_{j}f^{\alpha}_{kil})\nonumber\\
    =&P^{ijkl}f^{\alpha}_{jl}f^{\alpha}_{ik}.
\end{align}
By the divergence-free property \eqref{P-divfree} of $P^{ijkl}$ and \eqref{P-symmetri}, the first term:
\begin{align*}
    &\pt_iP^{ijkl}\pt_lg_{jk}\\
    =&(\nabla_iP^{ijkl}-P^{sjkl}\Gamma_{is}^i-P^{iskl}\Gamma_{is}^j-P^{ijsl}\Gamma_{is}^k-P^{ijks}\Gamma_{is}^l)\pt_lg_{jk}\\
    =&-P^{ijkl}(\Gamma_{is}^s\pt_lg_{jk}+\Gamma_{ik}^s\pt_lg_{js}+\Gamma_{il}^s\pt_sg_{jk})\\
    =&-P^{ijkl}\biggl(\Gamma_{is}^s(f^{\alpha}_{jl}f^{\alpha}_k+f^{\alpha}_{j}f^{\alpha}_{kl})+\Gamma_{ik}^s(f^{\alpha}_{jl}f^{\alpha}_s+f^{\alpha}_{j}f^{\alpha}_{sl}) +\Gamma_{il}^s(f^{\alpha}_{js}f^{\alpha}_k+f^{\alpha}_{j}f^{\alpha}_{ks})\biggr)\\
    =&-P^{ijkl}(\Gamma_{ik}^{s}f_{jl}^{\alpha}f_s^{\alpha}+\Gamma_{is}^sf_{jl}^{\alpha}f_k^{\alpha}+\Gamma_{il}^sf_{js}^{\alpha}f_k^{\alpha})\\
    =&I+II
\end{align*}
where
\begin{align*}
    I=&-P^{ijkl}\Gamma_{ik}^{s}f_{jl}^{\alpha}f_s^{\alpha}\\
    II=&-P^{ijkl}(\Gamma_{is}^sf_{jl}^{\alpha}f_k^{\alpha}+\Gamma_{il}^sf_{js}^{\alpha}f_k^{\alpha}).
\end{align*}
We will first calculate $I$:
\begin{align}
    I=&-P^{ijkl}\Gamma_{ik}^{s}f_{jl}^{\alpha}f_s^{\alpha}\nonumber\\
    =&-P^{ijkl}U^{\beta\gamma}f_s^{\beta}f_{ik}^{\gamma}f_{jl}^{\alpha}f_s^{\alpha}\nonumber\\
    =&P^{ijkl}f_{ik}^{\gamma}f_{jl}^{\alpha}U^{\beta\gamma}\langle Df^{\beta},Df^{\alpha}\rangle\nonumber\\
    =&P^{ijkl}f_{ik}^{\gamma}f_{jl}^{\alpha}(U^{\alpha\gamma}-\delta_{\alpha\gamma})\nonumber\\
    =&P^{ijkl}f_{ik}^{\gamma}f_{jl}^{\alpha}U^{\alpha\gamma}-P^{ijkl}f_{ik}^{\alpha}f_{jl}^{\alpha}\nonumber\\
    =&\frac 12L_2-P^{ijkl}f_{ik}^{\alpha}f_{jl}^{\alpha}.
\end{align}
Thus we have
\begin{align}
    \pt_i(P^{ijkl}\pt_lg_{jk})=&\frac 12L_2+\frac 12L_2^{\bot}
\end{align}
where
\begin{align}
    \frac 12L_2^{\bot}=II=&-P^{ijkl}(\Gamma_{is}^sf_{jl}^{\alpha}f_k^{\alpha}+\Gamma_{il}^sf_{js}^{\alpha}f_k^{\alpha})\nonumber\\
    =&P^{ijkl}(U^{\beta\gamma}f^{\beta}_sf^{\gamma}_{js}f_{il}^{\alpha}f_k^{\alpha}-U^{\beta\gamma}f_s^{\beta}f^{\gamma}_{il}f_{js}^{\alpha}f_k^{\alpha})\nonumber\\
    =&P^{ijkl}U^{\beta\gamma}f^{\beta}_sf_k^{\alpha}(f^{\gamma}_{js}f_{il}^{\alpha}-f^{\gamma}_{il}f_{js}^{\alpha}).
\end{align}
For simplicity, we denote
\begin{align}
    T_{ijkl}=&U^{\beta\gamma}f^{\beta}_sf_k^{\alpha}(f^{\gamma}_{js}f_{il}^{\alpha}-f^{\gamma}_{il}f_{js}^{\alpha})\nonumber\\
    =&f^{\alpha}_k\left(\langle A^{\gamma}\pt_j,\nabla f^{\gamma}\rangle\langle A^{\alpha}\pt_i,\pt_l\rangle-\langle A^{\alpha}\pt_j,\nabla f^{\gamma}\rangle\langle A^{\gamma}\pt_i,\pt_l\rangle\right).
\end{align}
Note that in general, $T_{ijkl}$ has no symmetric property. In the following five lemmas, we will calculate $\frac 12L_2^{\bot}=P^{ijkl}T_{ijkl}$ explicitly. Recall that
\begin{align*}
    P^{ijkl}=R^{ijkl}+R^{jk}g^{il}-R^{jl}g^{ik}-R^{ik}g^{jl}+R^{il}g^{jk}+\frac 12R(g^{ik}g^{jl}-g^{il}g^{jk}).
\end{align*}
In the following calculation, we will use \eqref{eq-shape},\eqref{eq-normal},\eqref{eq-Ricci} and \eqref{eq-commute} many times.

\begin{lem}\label{lem-Riem}
We have
\begin{align}
    &R^{ijkl}T_{ijkl}\nonumber\\
    =&U^{\mu\nu}\langle \left(A^{\mu}(A^{\nu}R^{\bot}_{\alpha\gamma}+R^{\bot}_{\nu\gamma}A^{\alpha}-R^{\bot}_{\nu\alpha}A^{\gamma})+Tr(A^{\alpha}A^{\mu}) R^{\bot}_{\gamma\nu}\right)\nabla f^{\alpha},\nabla f^{\gamma}\rangle
\end{align}
\end{lem}
\begin{proof}
From lemma \ref{lem-formulas}, we get
\begin{align*}
    R^{ijkl}=U^{\mu\nu}(f^{\mu}_{mn}f^{\nu}_{th}-f^{\mu}_{mh}f^{\nu}_{tn})g^{mi}g^{nk}g^{jt}g^{hl}.
\end{align*}
We have
\begin{align}
    &U^{\mu\nu}f^{\mu}_{mn}f^{\nu}_{th}g^{mi}g^{nk}g^{jt}g^{hl}T_{ijkl}\nonumber\\
    =&U^{\mu\nu}g^{nk}g^{hl}U^{\beta\gamma}f^{\beta}_sf_k^{\alpha}\biggl(\langle A^{\gamma}\pt_s,A^{\nu}\pt_h\rangle\langle A^{\alpha}\pt_l,A^{\mu}\pt_n\rangle-\langle A^{\alpha}\pt_s,A^{\nu}\pt_h\rangle\langle A^{\gamma}\pt_l,A^{\mu}\pt_n\rangle\biggr)\nonumber\\
    =&U^{\mu\nu}g^{hl}\biggl(\langle A^{\gamma}\nabla f^{\gamma},A^{\nu}\pt_h\rangle\langle A^{\alpha}\pt_l,A^{\mu}\nabla f^{\alpha}\rangle-\langle A^{\alpha}\nabla f^{\gamma},A^{\nu}\pt_h\rangle\langle A^{\gamma}\pt_l,A^{\mu}\nabla f^{\alpha}\rangle\biggr)\nonumber\\
    =&U^{\mu\nu}\biggl(\langle A^{\alpha}A^{\nu}A^{\gamma}\nabla f^{\gamma},A^{\mu}\nabla f^{\alpha}\rangle-\langle A^{\gamma}A^{\nu}A^{\alpha}\nabla f^{\gamma},A^{\mu}\nabla f^{\alpha}\rangle\biggr)\label{riem-eq1}
\end{align}
Further more, the commutation of shape operator gives the normal curvature terms.
\begin{align*}
    &\langle A^{\alpha}A^{\nu}A^{\gamma}\nabla f^{\gamma},A^{\mu}\nabla f^{\alpha}\rangle-\langle A^{\gamma}A^{\nu}A^{\alpha}\nabla f^{\gamma},A^{\mu}\nabla f^{\alpha}\rangle\\
    =&\langle (A^{\nu}A^{\alpha}+R^{\bot}_{\nu\alpha})A^{\gamma}\nabla f^{\gamma},A^{\mu}\nabla f^{\alpha}\rangle\\
    &-\langle (A^{\nu}A^{\gamma}+R^{\bot}_{\nu\gamma})A^{\alpha}\nabla f^{\gamma},A^{\mu}\nabla f^{\alpha}\rangle\\
    =&\langle (A^{\nu}R^{\bot}_{\gamma\alpha}+R^{\bot}_{\nu\alpha}A^{\gamma}-R^{\bot}_{\nu\gamma}A^{\alpha})\nabla f^{\gamma},A^{\mu}\nabla f^{\alpha}\rangle
\end{align*}
We conclude that the quantity \eqref{riem-eq1} is equal to
\begin{align*}
    \eqref{riem-eq1}=&\langle U^{\mu\nu}A^{\mu}\left(A^{\nu}R^{\bot}_{\gamma\alpha}+R^{\bot}_{\nu\alpha}A^{\gamma}-R^{\bot}_{\nu\gamma}A^{\alpha}\right)\nabla f^{\gamma},\nabla f^{\alpha}\rangle.
\end{align*}
Finally
\begin{align*}
    &-U^{\mu\nu}f^{\mu}_{mh}f^{\nu}_{tn}g^{mi}g^{nk}g^{jt}g^{hl}T_{ijkl}\\
    =&U^{\mu\nu}g^{nk}g^{hl}U^{\beta\gamma}f^{\beta}_sf_k^{\alpha}\biggl(\langle A^{\gamma}\pt_l,A^{\mu}\pt_h\rangle\langle A^{\alpha}\pt_s,A^{\nu}\pt_n\rangle-\langle A^{\alpha}\pt_l,A^{\mu}\pt_h\rangle\langle A^{\gamma}\pt_s,A^{\nu}\pt_n\rangle\biggr)\\
    =&U^{\mu\nu}g^{hl}\biggl(\langle A^{\gamma}\pt_l,A^{\mu}\pt_h\rangle\langle A^{\alpha}\nabla f^{\gamma},A^{\nu}\nabla f^{\alpha}\rangle-\langle A^{\alpha}\pt_l,A^{\mu}\pt_h\rangle\langle A^{\gamma}\nabla f^{\gamma},A^{\nu}\nabla f^{\alpha}\rangle\biggr)\\
    =&U^{\mu\nu}g^{hl}\langle A^{\alpha}\pt_l,A^{\mu}\pt_h\rangle\left(\langle A^{\gamma}\nabla f^{\alpha},A^{\nu}\nabla f^{\gamma}\rangle-\langle A^{\gamma}\nabla f^{\gamma},A^{\nu}\nabla f^{\alpha}\rangle\right)\\
    =&U^{\mu\nu}Tr(A^{\alpha}A^{\mu})\langle R^{\bot}_{\gamma\nu}(\nabla f^{\alpha}),\nabla f^{\gamma}\rangle.
\end{align*}
\end{proof}

\begin{lem}\label{lem-Ric4}
We have
\begin{align}
    R^{jk}g^{il}T_{ijkl}=&\langle (Tr A^{\alpha})T^{\bot}_{\gamma}\nabla f^{\alpha}, \nabla f^{\gamma}\rangle.
\end{align}
\end{lem}
\begin{proof}
Recall
\begin{align*}
   &R^{jk}g^{il}T_{ijkl}\\
   =&R^{jk}f^{\alpha}_k\left(\langle A^{\gamma}\nabla f^{\gamma},\pt_j\rangle Tr A^{\alpha}-\langle A^{\alpha}\nabla f^{\gamma},\pt_j\rangle Tr A^{\gamma}\right)\\
   =&Tr A^{\alpha}\left(\langle A^{\gamma}\nabla f^{\gamma},Rc(\nabla f^{\alpha})\rangle-\langle A^{\gamma}\nabla f^{\alpha},Rc(\nabla f^{\gamma})\rangle\right)\\
   =&Tr A^{\alpha}\left(\langle \nabla f^{\gamma},A^{\gamma}Rc(\nabla f^{\alpha})\rangle-\langle \nabla f^{\alpha},A^{\gamma}Rc(\nabla f^{\gamma})\rangle\right)\\
   =&Tr A^{\alpha}\left(\langle \nabla f^{\gamma},Rc A^{\gamma}(\nabla f^{\alpha})+T^{\bot}_{\gamma}\nabla f^{\alpha}\rangle-\langle \nabla f^{\alpha},A^{\gamma}Rc(\nabla f^{\gamma})\rangle\right)\\
   =&Tr A^{\alpha}\langle \nabla f^{\gamma},T^{\bot}_{\gamma}\nabla f^{\alpha}\rangle.
\end{align*}
\end{proof}

\begin{lem}\label{lem-Ric1}
We have
\begin{align}
    -R^{jl}g^{ik}T_{ijkl}=&-\langle \left(Rc R^{\bot}_{\alpha\gamma}+T^{\bot}_{\gamma}A^{\alpha}+A^{\alpha}T^{\bot}_{\gamma}\right)\nabla f^{\alpha},\nabla f^{\gamma}\rangle.
\end{align}
\end{lem}
\begin{proof}
\begin{align*}
   &R^{jl}g^{ik}T_{ijkl}\\
   =&R^{jl}\left(\langle A^{\gamma}\nabla f^{\gamma},\pt_j\rangle\langle A^{\alpha}\nabla f^{\alpha},\pt_l\rangle-\langle A^{\alpha}\nabla f^{\gamma},\pt_j\rangle\langle A^{\gamma}\nabla f^{\alpha},\pt_l\rangle\right)\\
   =&\langle A^{\alpha}\nabla f^{\alpha},RcA^{\gamma}(\nabla f^{\gamma})\rangle-\langle A^{\gamma}\nabla f^{\alpha},RcA^{\alpha}(\nabla f^{\gamma})\rangle\\
   =&\langle \nabla f^{\alpha},A^{\alpha}RcA^{\gamma}(\nabla f^{\gamma})\rangle-\langle Rc A^{\gamma}\nabla f^{\alpha},A^{\alpha}\nabla f^{\gamma}\rangle\\
   =&\langle \nabla f^{\alpha},(Rc A^{\alpha}+T^{\bot}_{\alpha})A^{\gamma}(\nabla f^{\gamma})\rangle\\
   &-\langle  (A^{\gamma}Rc-T^{\bot}_{\gamma}) \nabla f^{\alpha},A^{\alpha}\nabla f^{\gamma}\rangle\\
   =&\langle Rc \nabla f^{\alpha},R^{\bot}_{\gamma\alpha}\nabla f^{\gamma}\rangle+\langle \nabla f^{\alpha},T^{\bot}_{\alpha}A^{\gamma}(\nabla f^{\gamma})\rangle+\langle T^{\bot}_{\gamma} \nabla f^{\alpha},A^{\alpha}\nabla f^{\gamma}\rangle
\end{align*}

\end{proof}

\begin{lem}\label{lem-Ric2}
We have
\begin{align}
    -R^{ik}g^{jl}T_{ijkl}=&-\langle Rc R^{\bot}_{\alpha\gamma}\nabla f^{\alpha},\nabla f^{\gamma}\rangle.
\end{align}
\end{lem}
\begin{proof}
\begin{align*}
    &R^{ik}g^{jl}T_{ijkl}\\
    =&R^{ik}f^{\alpha}_k\left(\langle A^{\gamma}\nabla f^{\gamma},A^{\alpha}\pt_i\rangle-\langle A^{\alpha}\nabla f^{\gamma},A^{\gamma}\pt_i\rangle\right)\\
    =&R^{ik}f^{\alpha}_k\langle (A^{\alpha}A^{\gamma}-A^{\gamma}A^{\alpha})\nabla f^{\gamma},\pt_i\rangle\\
    =&\langle R^{\bot}_{\gamma\alpha}\nabla f^{\gamma},Rc(\nabla f^{\alpha})\rangle\\
    =&\langle Rc R^{\bot}_{\alpha\gamma}\nabla f^{\alpha},\nabla f^{\gamma}\rangle.
\end{align*}
\end{proof}

\begin{lem}\label{lem-R}
We have
\begin{align}
    &\left(R^{il}g^{jk}+\frac 12 R(g^{ik}g^{jl}-g^{il}g^{jk})\right)T_{ijkl}\nonumber\\
    =&\frac 12 R\langle R^{\bot}_{\alpha\gamma}(\nabla f^{\alpha}),\nabla f^{\gamma}\rangle.
\end{align}
\end{lem}
\begin{proof}[Proof of lemma \ref{lem-R}]
We calculate as following:
\begin{align*}
    &g^{jk}T_{ijkl}\\
     =&\langle A^{\alpha}\pt_i,\pt_l\rangle \langle A^{\gamma}(\nabla f^{\alpha}),\nabla f^{\gamma}\rangle-\langle A^{\gamma}\pt_i,\pt_l\rangle \langle A^{\alpha}(\nabla f^{\alpha}),\nabla f^{\gamma}\rangle\\
    =&\langle A^{\gamma}\pt_i,\pt_l\rangle \left(\langle A^{\alpha}(\nabla f^{\alpha}),\nabla f^{\gamma}\rangle-\langle A^{\alpha}(\nabla f^{\gamma}),\nabla f^{\alpha}\rangle\right)\\
    =&0,
\end{align*}
and
\begin{align*}
    &g^{ik}g^{jl}T_{ijkl}\\
    =&g^{jl}\left(\langle A^{\gamma}\pt_j,\nabla f^{\gamma}\rangle\langle A^{\alpha}(\nabla f^{\alpha}),\pt_l\rangle-\langle A^{\alpha}\pt_j,\nabla f^{\gamma}\rangle\langle A^{\gamma}(\nabla f^{\alpha}),\pt_l\rangle\right)\\
    =&\langle A^{\alpha}(\nabla f^{\alpha}),A^{\gamma}(\nabla f^{\gamma})\rangle-\langle A^{\gamma}(\nabla f^{\alpha}),A^{\alpha}(\nabla f^{\gamma})\rangle \\
    =&\langle R^{\bot}_{\alpha\gamma}(\nabla f^{\alpha}),\nabla f^{\gamma}\rangle.
\end{align*}
Then the lemma follows easily.
\end{proof}

Finally, from lemma \ref{lem-Riem} to lemma \ref{lem-R}, we get that
\begin{align}
    \frac 12L_2^{\bot}=&U^{\mu\nu}\langle \left(A^{\mu}(A^{\nu}R^{\bot}_{\alpha\gamma}+R^{\bot}_{\nu\gamma}A^{\alpha}-R^{\bot}_{\nu\alpha}A^{\gamma})+Tr(A^{\alpha}A^{\mu}) R^{\bot}_{\gamma\nu}\right)\nabla f^{\alpha},\nabla f^{\gamma}\rangle\nonumber\\
    &-\langle \left(2Rc R^{\bot}_{\alpha\gamma}+T^{\bot}_{\gamma}A^{\alpha}+A^{\alpha}T^{\bot}_{\gamma}\right)\nabla f^{\alpha},\nabla f^{\gamma}\rangle\nonumber\\
    &+\langle (Tr A^{\alpha})T^{\bot}_{\gamma}\nabla f^{\alpha}, \nabla f^{\gamma}\rangle+\frac 12 R\langle R^{\bot}_{\alpha\gamma}(\nabla f^{\alpha}),\nabla f^{\gamma}\rangle.\label{eq-L2}
\end{align}
We complete the proof of Proposition \ref{prop-mass}.
\endproof

\begin{rem}
From the above proof, we can also derive that
\begin{align}
    \pt_i(P_1^{ijkl}\pt_lg_{jk})=\frac 12R+\frac 12R^{\bot},
\end{align}
where $R^{\bot}=\langle R^{\bot}_{\alpha\gamma}(\nabla f^{\alpha}),\nabla f^{\gamma}\rangle$.
\end{rem}

\begin{proof}[Proof of Theorem \ref{thm2}]
From proposition \ref{prop-mass} and the divergence theorem, we have
\begin{align*}
    m_2=&\lim_{r\rightarrow\infty}c_2(n)\int_{S_r}P^{ijkl}\pt_lg_{jk}\nu_idS_r\\
    =&c_2(n)\int_{\R^n}\pt_i(P^{ijkl}\pt_lg_{jk})d\mu_{\R^n}\\
    =&\frac {c_2(n)}2\int_{\R^n}(L_2+L_2^{\bot})d\mu_{\R^n}\\
    =&\frac {c_2(n)}2\int_{\R^n}(L_2+L_2^{\bot})\frac{1}{\sqrt{G}}d\mu_M.
\end{align*}
The last equality is due to the fact that
\begin{align*}
    d\mu_{\R^n}=\frac{1}{\sqrt{G}}d\mu_M,
\end{align*}
where $G$ is the determinant of the metric $(g_{ij})$.
\end{proof}

\section{The proof of Theorem \ref{thm3}}
The proof of Theorem \ref{thm3} is similar as in \cite{GWW} with a slight modification. For convenience of readers, we give a complete proof. By the equation \eqref{key-equ1}, integrating by parts gives that
\begin{align*}
m_2&=\f{1}{2}c_2(n)\int_M (L_2+L_2^{\perp})\f{1}{\sqrt{G}}d\mu_M-c_2(n)\int_\Sigma P^{ijkl}\pt_l g_{jk}\nu_i d\mu_{\Sigma}\\
   &=\f{1}{2}c_2(n)\int_M (L_2+L_2^{\perp})\f{1}{\sqrt{G}}d\mu_M-c_2(n)\int_\Sigma P^{ijkl}f^\alpha_{lj}f^\alpha_k \nu_i d\mu_{\Sigma},
\end{align*}
where $\nu$ is the inward normal of $\Sigma$.

By assumption each connected component of $\Sigma$ is a level set of $f^\alpha$, we know that in $\Sigma$
\begin{equation*}
Df^\alpha=\la Df^\alpha,\nu\ra \nu.
\end{equation*}

Thus if we rotate the coordinates such that $\nu=-e_1$ at a given point in $\Sigma$, then at that point there hold
\begin{equation}
f^\alpha_a=0 \textrm{ and } f^\alpha_{ab}=A_{ab}|Df^\alpha|=A_{ab}|f^\alpha_1|,
\end{equation}
where $A_{ab}$ is the second fundamental form of the isometric embedding $(\Sigma,h)$ into the Euclidean space $\R^n$ with $h$ the induced metric on $\Sigma$. In this section, we use the convention that $\{a,b,c,\cdots\}$ stand for $\{2,3,\cdots,n\}$ and $\{j,k,l,\cdots\}$ stand for $\{1,2,\cdots,n\}$.

Moreover, the nonzero components of the metric $g$ of $M$ are
\begin{align*}
g_{11}=1+|Df|^2,& \quad g_{aa}=1,\\
g^{11}=\f{1}{1+|Df|^2},&\textrm{ and } g^{aa}=1.
\end{align*}

Using the expression of $P^{ijkl}$, we have
\begin{align*}
P^{ijkl}f^\alpha_{lj}f^\alpha_k \nu_i =&R^{ijkl}f^\alpha_{lj}f^\alpha_k \nu_i+R^{jk}g^{il}f^\alpha_{lj}f^\alpha_k \nu_i\\
                                      &-R^{jl}g^{ik}f^\alpha_{lj}f^\alpha_k \nu_i-R^{ik}g^{jl}f^\alpha_{lj}f^\alpha_k \nu_i\\
                                      &+R^{il}g^{jk}f^\alpha_{lj}f^\alpha_k \nu_i+\f{1}{2}R(g^{ik}g^{jl}-g^{il}g^{jk})f^\alpha_{lj}f^\alpha_k \nu_i\\
                                      =&I+II+III+IV+V+VI.
\end{align*}

Now we calculate these six terms as follows.
\begin{align*}
I &=R^{1a1b}f^\alpha_{ab}f^\alpha_1 \nu_1=-R^{1a1b}A_{ab}(f^\alpha_1)^2=-R^{1a1b}A_{ab}|Df|^2,\\
II&=R^{1a}g^{11}f^\alpha_{1a}f^\alpha_1\nu_1=-R^{1a}f^\alpha_{1a}\f{|f^\alpha_1|}{1+|Df|^2},\\
III&=-R^{jl}g^{11}f^\alpha_{lj}f^\alpha_1 \nu_1\\
&=(2R^{1l}f^\alpha_{1l}+R^{ab}A_{ab}|f^\alpha_1|-R^{11}f^{\alpha}_{11})\f{|f^\alpha_1|}{1+|Df|^2},\\
IV&=-R^{11}g^{jl}f^\alpha_{lj}f^\alpha_1\nu_1=R^{11}(H|Df|^2+\f{f^\alpha_{11}|f^\alpha_1|}{1+|Df|^2}),\\
V&=R^{1l}g^{11}f^\alpha_{1l}f^\alpha_1\nu_1=-R^{1l}f^\alpha_{1l}\f{|f^\alpha_1|}{1+|Df|^2},\\
VI&=\f{1}{2}R(g^{11}g^{jl}f^\alpha_{lj}f^\alpha_1 \nu_1-g^{1l}g^{1j}f^\alpha_{lj}f^\alpha_1\nu_1)\\
&=-\f{1}{2}RH\f{|Df|^2}{1+|Df|^2},
\end{align*}
where $H$ is the mean curvature of embedding $(\Sigma,h)$ in $\R^n$. Putting these together, we obtain
\begin{align}
P^{ijkl}f^\alpha_{lj}f^\alpha_k \nu_i=&-R^{1a1b}A_{ab}|Df|^2+R^{11}H|Df|^2\nonumber\\
&\quad +R^{ab}A_{ab}\f{|Df|^2}{1+|Df|^2}-\f{1}{2}RH\f{|Df|^2}{1+|Df|^2} \label{eq1}.
\end{align}

Since $\Sigma$ is the level set of $f$, we can identify $\Sigma$ with its graph. Thus $(\Sigma,h)$ is also an isometric embedding in $(M,g)$. Denote $\tilde{A}_{ab}$ the second fundamental form of the embedding $(\Sigma,h) \hookrightarrow (M^n,g)$. Then a direct calculation gives that
\begin{equation}
\tilde{A}_{ab}=\f{1}{\sqrt{1+|Df|^2}}A_{ab}.
\end{equation}

\begin{rem}
From the above formula, we can see that if $|Df|=+\infty$ on $\Sigma$, then $\Sigma$ is area-minimizing horizon . In the next section, we will show that under the condition of Theorem \ref{thm4}, $|Df|=+\infty$ on $\Sigma$.
\end{rem}

Denote $\widehat{Rm}$ the Riemannian curvature of $\Sigma$. Then by the Gauss equation we obtain
\begin{align*}
\hat{R}^{abcd}&=A^{ac}A^{bd}-A^{ad}A^{bc}
\end{align*}
and also
\begin{align*}
    \hat{R}^{abcd}&=R^{abcd}+\tilde{A}^{ac}\tilde{A}^{bd}-\tilde{A}^{ad}\tilde{A}^{bc}.
\end{align*}
Thus
\begin{equation}
R^{abcd}=\f{|Df|^2}{1+|Df|^2}\hat{R}^{abcd}.
\end{equation}

Now we can write down
\begin{align*}
R^{11}&=R^{a1a1}g_{aa}\\
      &=R^{a1a1},\\
R^{ab}&=R^{acbd}g_{cd}+R^{a1b1}g_{11}\\
      &=R^{acbc}+R^{a1b1}(1+|Df|^2),\\
R     &=R^{11}g_{11}+R^{ab}g_{ab}\\
      &=2R^{1a1a}(1+|Df|^2)+R^{acac}.
\end{align*}

Plugging these terms into \eqref{eq1}, we have
\begin{align*}
P^{ijkl}f^\alpha_{lj}f^\alpha_k \nu_i&=-R^{1a1b}A_{ab}|Df|^2+R^{a1a1}H|Df|^2\\
                                     &+(R^{acbc}+R^{a1b1}(1+|Df|^2))A_{ab}\f{|Df|^2}{1+|Df|^2}\\
                                     &-\f{1}{2}(2R^{1a1a}(1+|Df|^2)+R^{acac})H\f{|Df|^2}{1+|Df|^2}\\
                                     &=\f{|Df|^2}{1+|Df|^2}(R^{acbc}A_{ab}-\f{1}{2}HR^{abab})\\
                                     &=(\f{|Df|^2}{1+|Df|^2})^2(\hat{R}^{ab}-\f{1}{2}\hat{R}h^{ab})A_{ab}\\
                                     &=-3(\f{|Df|^2}{1+|Df|^2})^2H_3,\\
\end{align*}
which completes the proof of Theorem \ref{thm3}.

\section{The proofs of Theorem \ref{thm4} and Theorem \ref{thm-penrose}}

\begin{proof}[Proof of Theorem \ref{thm4}]
Since $M$ extends to the boundary in a $C^2$ fashion, the limit $\lim\limits_{x\rightarrow\pt\Omega}|Df|^2$ exists (could be $+\infty$). In fact, we claim that under the condition of Theorem \ref{thm4}, the limit
\begin{equation}\label{eq4}
\lim\limits_{x\rightarrow \pt \Omega} |Df|^2=+\infty.
\end{equation}

To prove this claim, we fix a point $p\in \pt \Omega$. Suppose on the contrary that $\lim\limits_{x\rightarrow p} |Df|^2<+\infty$. At the point $p$, as in the last section, we can rotate the coordinates such that $\{e_2,\cdots,e_n\}$ forms the basis of the tangent space $T_p\Sigma$.

Now by composing an orthogonal transformation of $\R^m$ with $f:\R^n\setminus \Omega\rightarrow \R^m$, we can assume that the straight line $l=\{(0;t e_1)\in \R^{n+m}|t\in \R\}$ without loss of generality. Since the manifold $\bar{M}$ is tangent to the cylinder $\Sigma \times l$, we know that $\{(0;e_1),(e_2;0),\cdots,(e_n;0)\in \R^{n+m}\}$ is a basis of $T_pM$.

On the other hand, recall that at any point where $|Df|^2<+\infty$, i.e. $f^\alpha_i$ is finite, $\{\pt_i=(e_i; f^\alpha_i e_\alpha)\in \R^{n+m}\}$ is another basis of $T_pM$. But these two bases of $T_pM$ can not be represented mutually, which is a contradiction. Therefore, we prove the claim \eqref{eq4}. From this claim, Theorem \ref{thm4} follows immediately.
\end{proof}

\begin{proof}[Proof of Theorem \ref{thm-penrose}]
 Assume each connected component $\Sigma_i$ of the boundary $\pt\Omega$ is $3$-convex and star-shaped, the Alexandrov-Fenchel ienquality \cite{GL} (see also \cite{CW}) gives that
\begin{align}\label{eq-5-1}
3c_2(n)\int_{\Sigma_i}H_3d\mu_{\Sigma_i}\geq& \frac 14(\frac {|\Sigma_i|}{\omega_{n-1}})^{\frac {n-4}{n-1}}.
\end{align}
The equality holds if and only if $\Sigma_i$ is a sphere. Recall that for nonnegative numbers $a_1,\cdots,a_k$ and $0\leq\beta\leq 1$, we have the following algebraic inequality (see \cite{HW2})
\begin{align}
    \sum_i^ka_i^{\beta}\geq (\sum_i^ka_i)^{\beta}.
\end{align}
If $0\leq\beta<1$, the equality holds if and only if at most one of $a_i$ is nonzero. Thus by summing \eqref{eq-5-1}, we obtain that
 \begin{align*}
3c_2(n)\int_{\Sigma}H_3d\mu_{\Sigma_i}\geq& \sum_i\frac 14(\frac {|\Sigma_i|}{\omega_{n-1}})^{\frac {n-4}{n-1}}\\
\geq&\frac 14(\frac {|\Sigma|}{\omega_{n-1}})^{\frac {n-4}{n-1}}
\end{align*}
Combining with Theorem \ref{thm4}, we get the Penrose inequality \eqref{eq-penrose}. Moreover, if the equality holds, we have that the second Gauss-Bonnet curvature $L_2$ vanishes on $M$ and $\Sigma$ has only one connected component and is a sphere.
\end{proof}

\section{Positive mass theorem and Penrose inequality in Einstein-Gauss-Bonnet gravity}

In \cite{GWW1}, Ge-Wang-Wu obtained a positive mass theorem for asymptotically flat codimension one graph under the condition $R+\alpha L_2\geq 0$ and a Penrose type inequality in the case $\alpha>0$, where $\alpha$ is a constant.  In this section, we give a generalization of their result to arbitrary codimension graph.

Let $n\geq 4$ and $(M^n,g)$ be an asymptotically flat manifold of decay order $\tau>\frac{n-2}2$ and $R+\alpha L_2$ is integrable in $(M,g)$. The Einstein-Gauss-Bonnet mass energy is defined by (\cite{GWW1}, see also \cite{DT1,DT2})
\begin{align}\label{mass-EGB}
    m_{EGB}=\frac 1{(n-1)\omega_{n-1}}\lim_{r\rightarrow\infty}\int_{S_r}(P_1^{ijkl}+\alpha P_2^{ijkl})\pt_lg_{jk}\nu_idS.
\end{align}
Ge-Wang-Wu \cite{GWW1} proved that $m_{EGB}$ is well-defined and is a geometric invariant on $M$.

\begin{thm}[Positive mass theorem]
Suppose $(M^n,g)$ is a graph of a smooth asymptotically flat map $f:\R^n\rightarrow\R^m$ of order $\tau>\frac {n-2}2$. Then we have
\begin{align*}
    m_{ADM}=m_{EGB}=&\frac 1{2(n-1)\omega_{n-1}}\int_{M}(R+\alpha L_2+R^{\bot}+\alpha L_2^{\bot})\frac 1{\sqrt{G}}d\mu_M,
\end{align*}
where $R^{\bot}=\langle R^{\bot}_{\alpha\gamma}(\nabla f^{\alpha}),\nabla f^{\gamma}\rangle$ and $L_2^{\bot}$ is expressed as in \eqref{eq-L2}. In particular, when $M$ has flat normal bundle as a submanifold of $\R^{n+m}$ and $R+\alpha L_2\geq 0$, we have
\begin{align}
    m_{ADM}\geq 0.
\end{align}
\end{thm}
\proof
Since $\frac {n-2}2>\frac {n-4}2$, under the condition $\tau>\frac {n-2}2$, we have
\begin{align*}
    \lim_{r\rightarrow\infty}\int_{S_r}P_2^{ijkl}\pt_lg_{jk}\nu_idS=\lim_{r\rightarrow\infty}O(r^{-2\tau+n-4})=0.
\end{align*}
Thus $m_{ADM}=m_{EGB}$. From the proof of Proposition \ref{prop-mass}, one can also get
\begin{align}
    \pt_i(P_1^{ijkl}\pt_lg_{jk})=\frac 12R+\frac 12R^{\bot}.
\end{align}
Therefore
\begin{align}
    \pt_i((P_1^{ijkl}+\alpha P_2^{ijkl})\pt_lg_{jk})=\frac 12(R+\alpha L_2+R^{\bot}+\alpha L_2^{\bot}).
\end{align}
The rest of the proof is similar as Theorem \ref{thm2}, which we omit here.
\endproof

Similarly as in section 4 and section 5, we can prove
\begin{thm}[Penrose inequality]
Let $\Omega\subset\R^n$ be an open subset. Let $f:\R^n \setminus \Omega \rightarrow \R^m$ be a continuous map that is constant along each connected component of the boundary $\Sigma=\pt \Omega$ and asymptotically flat in $\R^n\setminus \bar{\Omega}$ of order $\tau>\frac {n-2}2$. Assume that the graph $M$ of $f$ extends $C^2$ to its boundary $\pt M$. Assume further that along each connected component $\Sigma_i$ of $\pt M$, the manifold $\bar{M}$ is tangent to the cylinder $\Sigma \times l_i$, where $l_i$ is a straight line of $\R^m$. Then
\begin{align*}
m_{ADM}=&\f1{2(n-1)\omega_{n-1}}\int_M (R+\alpha L_2+R^{\bot}+\alpha L_2^{\bot})\frac 1{\sqrt{G}}d\mu_M\\
&+\f{1}{2(n-1)\omega_{n-1}}\int_\Sigma (H+ 6\alpha H_3) d\mu_{\Sigma},
\end{align*}
where $H$ and $H_3$ are mean curvature and the $3$-th mean curvature of the embedding $\Sigma$ in $\R^n$. Moreover, if $M$ has nonnegative $R+\alpha L_2$ and flat normal bundle,  each connected components of the boundary $\Sigma=\pt \Omega$ is $3$-convex and star-shaped, $\alpha\geq 0$, then
\begin{align}\label{eq-penrose}
  m_{ADM}\geq &\frac 12(\frac {|\Sigma|}{\omega_{n-1}})^{\frac {n-2}{n-1}}+\frac {\alpha}2(n-2)(n-3)(\frac {|\Sigma|}{\omega_{n-1}})^{\frac {n-4}{n-1}}.
\end{align}
\end{thm}

\bibliographystyle{Plain}

\begin{thebibliography}{10}
\bibitem{ADM} R. Arnowitt, S.Deser and C. W. Misner, {\it Coordinate invariance and energy expressions in general relativity,} Phys. Rev. (2) \textbf{122} (1961), 997-1006.

\bibitem{Ba} R. Bartnik, {\it The mass of an asymptotically flat manifold,} Comm. Pure Appl. Math., \textbf{34} (1986) 661-693.

\bibitem{BL} H. L. Bray and D. A. Lee, {\it On the Riemannian Penrose inequality in dimensions less than eight}, Duke Math.J. \textbf{148} (2009), no. 1, 81-106.

\bibitem{Br} H. L. Bray, {\it Proof of the Riemannian Penrose inequality using the positive mass theorem}, J. Differential Geom. \textbf{59} (2001), no. 2, 177-267.
\bibitem{Br1} H. L. Bray, {\it On the positive mass, Penrose, an ZAS inequalities in general dimension}, Surveys in Geometric
Analysis and Relativity, Adv. Lect. Math. (ALM), 20, Int. Press, Somerville, MA, 2011

\bibitem{CW} S.-Y. A. Chang and Y.Wang, {\it On Aleksandrov-Fenchel inequalities for k-convex domains}, Milan J. Math., \textbf{79}
(2011), no. 1, 13-38

\bibitem{Da} S. Davis, {\it Generalized Israel junction conditions for a Gauss-Bonnet brane world}, Phys. Rev. D 67 (2003)
024030.

\bibitem{deLG1}L. de Lima  and F.Gir\~{a}o, {\it The ADM mass of asymptotically flat hypersurfaces}. To appear
in Trans. A.M.S. arXiv:1108.5474.

\bibitem{deLG2}L. de Lima  and F.Gir\~{a}o, {\it A rigidity result for the graph case of the Penrose inequality}
arXiv:1205.1132.

\bibitem{DT1} S. Deser, B. Tekin, {\it Gravitational energy in quadratic-curvature gravities,} Physical review letters, 89 (2002)
101101.

\bibitem{DT2} S. Deser, B. Tekin, {\it Energy in generic higher curvature gravity theories}, Physical Review D, 75 (2003) 084032

\bibitem{GWW} Y. Ge, G. Wang and J. Wu, {\it A new mass for asymptotically flat manifolds}, arXiv:1211.3645.

\bibitem{GWW1} Y. Ge, G.Wang and J. Wu, {\it A positive mass theorem in the Einstein-Gauss-Bonnet theory, }arXiv:1211.7305

\bibitem{GWW2} Y. Ge, G. Wang and J. Wu, {\it The Gauss-Bonnet-Chern mass of conformally flat manifolds. } arXiv:1212.3213

\bibitem{GL} P. Guan and J. Li, {\it The quermassintegral inequalities for k-convex starshaped domains,} Adv. Math. \textbf{221}(2009), 1725-1732.

\bibitem{HW1} L.-H. Huang  and D. Wu,  {\it Hypersurfaces with nonnegative scalar curvature}. arXiv:1102.5749.

\bibitem{HW2}  L.-H. Huang  and D. Wu,  {\it The equality case of the Penrose inequality for asymptotically flat
graphs}. arXiv:1205.2061.

\bibitem{HI} G. Huisken  and T. Ilmanen,  {\it The inverse mean curvature flow and the Riemannian Penrose
inequality.} J. Differential Geom. \textbf{59} (2001), no. 3, 353-437.

\bibitem{Lam}  M.-K. G.Lam, {\it The graph cases of the Riemannian positive mass and Penrose inequality in all dimensions,}
arXiv.org/1010.4256.

\bibitem{Mars} M. Mars. {\it Topical review: Present status of the Penrose inequality,} Classical and Quantum Gravity,
26(19):193001, Oct. 2009.

\bibitem{MV} H. Mirandola and F. Vitorio, {\it The positive mass Theorem and Penrose inequality for graphical manifolds},	arXiv:1304.3504.

\bibitem{SY79}   R. Schoen and S.-T. Yau , {\it On the proof of the positive mass conjecture in general relativity.}
Comm. Math. Phys. \textbf{65} (1979), no. 1, 45-76.

\bibitem{SY80} R. Schoen  and S.-T. Yau , {\it Proof of the Positive Mass Theorem II,} Comm. Math. Phys. \textbf{79}
(1981) 231-260.

 \bibitem{SY88}  R. Schoen and S. T. Yau, {\it Conformally flat manifolds, Kleinian groups and scalar curvature, } Invent. math.,\textbf{92} (1988),47-71.

\bibitem{Wit}   E. Witten, {\it A new proof of the positive energy theorem}, Comm. Math. Phys., \textbf{80} (1981) 381-402.

\end{thebibliography}

\end{document}